\documentclass[12pt,a4paper,twoside]{amsart}
\usepackage{a4wide,times, amsmath,amsbsy,amsfonts,amssymb,
stmaryrd,mathrsfs,graphicx, amscd, tikz-cd, cancel}
\usepackage{dsfont}
\usepackage[all]{xy}
\usepackage{xypic}
\usepackage{yhmath}
\usepackage{mathrsfs}
\usepackage{graphicx}
\usepackage{subfigure}
\usepackage{fancyhdr}
\usepackage{lastpage}  
\usepackage[normalem]{ulem}
\usetikzlibrary{matrix,arrows,decorations.pathmorphing}
\usepackage[active]{srcltx}
\usepackage[
	hypertexnames=false,
	hyperindex,
	pagebackref,
	breaklinks=true,
	bookmarks=true,
	colorlinks,
	linkcolor=blue,
	citecolor=red,
	urlcolor=red,
]{hyperref}
\usepackage[mathscr]{eucal}

\newcommand\SL{\operatorname{SL}}

\newcommand\SO{\operatorname{SO}}

\newcommand\PSL{\operatorname{PSL}}

\newcommand\Sp{\operatorname{Sp}}

\newcommand\edge{\textrm{edge}}
\newcommand\trc{\textrm{trc}}
\newcommand\la{{\langle}}
\newcommand\ra{{\rangle}}

\newcommand\End{\operatorname{End}}

\newcommand\Ext{\operatorname{Ext}}
\newcommand\Hom{\operatorname{Hom}}

\newcommand\Img{\operatorname{im}}

\newcommand\Gal{\operatorname{Gal}}

\newcommand\Acal{\mathcal A}

\newcommand\Scal{\mathcal S}

\newcommand\Csr{\mathscr C}
\newcommand\Bsr{\mathscr B}
\newcommand\Hsr{\mathscr H}
\newcommand\Isr{\mathscr I}

\newcommand\Ksr{\mathscr K}
\newcommand\Lsr{\mathscr L}
\newcommand\Osr{\mathscr O}

\newcommand\Psr{\mathscr P}

\newcommand\Wsr{\mathscr W}

\newcommand\Cds{\mathds{C}}

\newcommand\Fds{\mathds{F}}
\newcommand\Hds{\mathds{H}}

\newcommand\Pds{\mathds{P}}
\newcommand\Qds{\mathds{Q}}
\newcommand\Rds{\mathds{R}}

\newcommand\Zds{\mathds{Z}}

\newtheorem{theorem}{Theorem}[section]
\newtheorem{lem}[theorem]{Lemma}  
\newtheorem{thm}[theorem]{Theorem}  
\newtheorem{cor}[theorem]{Corollary}  
\newtheorem{prop}[theorem]{Proposition}

\theoremstyle{definition}
\newtheorem{rmk}[theorem]{Remark}
\newtheorem{eg}[theorem]{Example}

\begin{document}
\title{Monodromy and period map of the Winger Pencil}
\author{Eduard Looijenga and Yunpeng Zi}
\thanks{Part of the research for this paper was done when both authors were supported by the NSFC}
\address{Mathematisch Instituut, Universiteit Utrecht (Nederland) and Mathematics Department, University of Chicago (USA)}
\email{e.j.n.looijenga@uu.nl}
\address{Yau Mathematical Sciences Center, Tsinghua University, Beijing (China) and Yanqi Lake Beijing Institute of Mathematical Sciences and Applications, Huairou District, Beijing (China)}
\email{galoisexp@outlook.com}
\date{}

\begin{abstract}
	The sextic plane curves that are invariant under the standard action of the icosahedral group on the projective plane make up a pencil of genus ten curves  (spanned by a sum of six lines and a three times a  conic). This pencil was first considered in a note by R.~M.~Winger  in 1925 and is nowadays named after him. The second author recently gave this a modern treatment and proved among other things that it contains essentially every smooth genus  ten  curve with  icosahedral  symmetry. We here show that the Jacobian of such a curve contains  the tensor product of an elliptic curve  with a certain integral representation of the  icosahedral   group. We find that the elliptic curve  comes with  a distinguished point of order $3$, prove that the monodromy on this part of the homology  is the full congruence subgroup $\Gamma_1(3)\subset \SL_2(\Zds)$ and  subsequently  identify the base of the pencil with the associated modular curve.

	We also observe that the Winger pencil `accounts' for the deformation of the Jacobian of Bring's curve as a principal abelian fourfold with an action of the icosahedral group.


\end{abstract}

\maketitle

\section{Introduction}
The Winger family is the family of genus 10 curves endowed with faithful action of icosahedral group $\Isr$ (which is isomorphic to the alternating group $\Acal_5$, see Remark \ref{rem:symb} below) introduced in \cite{winger1925invariants}.
The second author showed in \cite{zi2021geometry} that the moduli stack of genus 10 curves endowed with faithful action of the icosahedral group $\Isr$ has two connected components that are exchanged by an outer automorphism of $\Isr$. Each connected component is given by a pencil, classically known as the \emph{Winger pencil}.
That pencil has four singular fibers, one of which is a conic with multiplicity $3$ that is after a base change replaceable by a smooth genus ten curve with an automorphism group that strictly contains the copy of $\Isr$, whereas the remaining three curves, a sum of six lines, an irreducible curve with ten nodes and an irreducible curve with six nodes are all stable.

It was there also proved that if $C$ is a smooth member of this pencil, then in  the $\Cds\Isr$-module
$H_1(C;\Cds)$ only two types of irreducible representations  appear, one of which is given by the restriction to $\Isr$ of the
reflection representation of the symmetric group $\Scal_5$. This reflection representation, which we denote  by $V$, is of dimension $4$ and  appears in $H_1(C;\Cds)$ with multiplicity two.
It has a natural integral model $V_o$ (which we describe in \ref{inforVo}) for which the isogeny lattice
$\Hom_{\Zds\Isr}(V_o,H_1(C))$ is free abelian of rank two. This defines a summand $\rho_V$ of the monodromy which takes its values in the special linear group of $\Hom_{\Zds\Isr}(V_o,H_1(C))$ (so that is a copy of $\SL_2(\Zds)$). The Hodge decomposition of $H^1(C)$
determines one of $\Hom_{\Zds\Isr}(V_o,H_1(C))$ with $(1,0)$-part and $(0,1)$-part both of dimension one and thus we have associated with $C$ an elliptic curve.

The monodromy representation $\rho_V$  has the remarkable  property that it is nontrivial around  3 (of the 4) singular fibers only and is of finite order (namely $3$) near one of them.  This observation is subsumed by our main theorem below. In it appears the congruence subgroup $\Gamma_1(3)$ of integral matrices
$(\begin{smallmatrix} a & b\\ c& d  \end{smallmatrix})\in \SL_2(\Zds)$ with $c\equiv 0\pmod{3}$. This group has
index $8$ in $\SL_2(\Zds)$  and since it does not contain $-1$, its image  in $\PSL_2(\Zds)=\SL_2(\Zds)/\{\pm1\}$ has index $4$. The associated modular curve $X_1(3):=\Gamma_1(3)\backslash \Hds$
is known to have exactly one orbifold point of order three and is  completed by  two cusps, one of width $1$ and the other of width $3$.
(If $d$ is a positive integer and  $\Gamma$ a subgroup of $\SL_2(\Zds)$, then a \emph{cusp of width} $d$ of $\Gamma$ is a primitive conjugacy class of $\Gamma$ that is contained in the  $\SL_2(\Zds)$-conjugacy class  of $(\begin{smallmatrix} 1 & d\\ 0& 1  \end{smallmatrix})\in \SL_2(\Zds)$.)


\begin{thm}\label{mtm2}
	The image of $\rho_V$ is conjugate to the congruence subgroup $\Gamma_1(3)$.
	Via the above construction the period map becomes  an isomorphism of  the base of the Winger pencil onto the completion of the modular curve
	$X_1(3)$. Under this isomorphism, the point defining the triple conic goes to the orbifold point,  the point representing the irreducible curve with six nodes
	(whose normalization is Bring's  curve) to another point of $X_1(3)$, the point defining the six lines to the cusp of width $3$ and
	the irreducible curve with ten nodes  to the cusp of width $1$.
\end{thm}

Let us elaborate on the appearance of Bring's curve. First, recall that this is a  curve of genus four that comes endowed with a faithful
action of the full symmetric group $\Scal_5$ and is unique for that property. Its canonical embedding (in a hyperplane of  $\Pds^4$) is
given as the common zero set of the first three symmetric functions in five variables. As explained in Remark 3.6 of \cite{zi2021geometry},
it contains an $\Scal_5$-orbit of size $24$ that decomposes into two $\Isr$-orbits of size $12$, each of which further decomposes in a
$\Isr$-invariant manner into six pairs (as an $\Isr$-set this is isomorphic to the set of vertices of a regular icosahedron which indeed make up
six antipodal pairs). If we take such an $\Isr$-orbit and identify the points of each of its 6 pairs, then we get the stable genus ten curve
that appears in the Winger pencil. Since the part of the Jacobian that we consider here survives  in the normalization of  this curve
(so in Brings's curve), our period map ignores its six double points.

This can also be expressed as follows: Bring's curve has no deformations as a $\Scal_5$-curve,
not even as a $\Isr$-curve. However the associated stable genus ten curve admits a smoothing which makes the Jacobian of Bring's curve deform as a factor of the Jacobians of a stable family of genus  ten curves with $\Isr$-symmetry. Thus the Winger pencil also provides a fitting coda to a story that began with the observation  of Riera and Rodr\'{i}guez \cite{riera1992period} that despite the rigid nature of Bring's curve,  its Jacobian  deforms in a one-parameter family of principally polarized abelian varieties with $\Isr$-action. Let us also mention here that Gonz\'ales-Aguilera and Rodr\'{i}guez  \cite{gonzalez2000families} determined the Jacobian of Bring's curve as a product of four elliptic curves, all isogenous to each other, which subsequently was made more precise by Braden-Northover \cite{braden2012bring}.
\\

A central role in the proof is played by a combinatorial model of a genus ten curve with $\Isr$-action. It is obtained by taking a regular dodecahedron that is truncated in a $\Isr$-invariant manner by removing at each vertex a small triangular neighborhood (so that we get ten antipodal pairs of triangles as boundary components) and subsequently  identifying opposite boundary triangles  by means  the antipodal involution. The resulting surface is oriented and of genus ten  and  comes with a piecewise euclidean structure that is $\Isr$-invariant.  This gives rise to a $\Isr$-invariant conformal and hence complex structure. We thus obtain a family  of Riemann surfaces with $\Isr$-action depending on one real parameter  (namely the ratio of the length of an edge of a truncation and the length of an edge of the dodecahedron). We analyse what happens when this ratio tends to its infimum ($0$) or its maximum ($1$). Both represent stable degenerations  and remarkably this suffices for computing the part of the monodromy $\rho_V$.
\\

After we posted the first version of this paper,  Harry Braden drew our attention to a 1995 paper by R.~H.~Dye \cite{dye1995plane}, in which the Winger pencil appears (as the display labeled (15) on page 100). Dye, who was apparently not aware of the work of Winger, points out that that this pencil has a member whose normalization is Bring's curve.

\begin{rmk}\label{rem:symb}
	In this paper we fix an \emph{oriented} euclidean $3$-space $I_{\Rds}$ and a regular dodecahedron $D\subset I_{\Rds}$
	centered at the origin.  
	The group of isometries of $D$ contains the antipodal involution (denoted here by $\iota$)  which reverses orientation.
	So if $\Isr\subset \SO(I_{\Rds})$ stands for the group orientation preserving isometries of $D$, then $\Isr\times\{1, \iota\}$ is the
	full group of isometries of $D$.

	The group $\Isr$ is isomorphic to $\Acal_5$. In fact, the collection $\Ksr$ of inscribed cubes of $D$ consists of $5$ elements and $\Isr$ induces the full group of even permutations of $\Ksr$, so that
	by numbering its elements  we obtain an isomorphism $\Isr\cong\Acal_5$. But in this paper it is the group $\Isr$ rather than $\Acal_5$ which comes up naturally and as there is for us no good reason to number the elements of  $\Ksr$, we will express our results in terms of
	$\Isr$ instead of $\Acal_5$.
\end{rmk}

\section{Preliminaries}

\subsection{Isotypical decomposition of a symplectic module over a group algebra}\label{subsect:isotypical}
We  begin with a brief review of  the basic theory, referring to  \cite{isaacs1994character} for more details. Let $G$ be a finite group and  $k$ a field of  characteristic zero which is one of the following: a fixed number field, $\Rds$ or $\Cds$.  Let $\chi(kG)$ be the set of irreducible characters of $kG$. It is a set of $k$-valued class functions on $G$ whose elements are invariant under the Galois group $\Gal(\overline k/k)$.

For $\lambda\in \chi (kG)$, we denote by $V_\lambda$  an irreducible $k G$-module that represents it.
By the Schur Lemma the ring $\End_{k G}(V_{\lambda})$ is a division algebra. Denote by $D_{\lambda}$ its opposite so that we can regard
$V_{\lambda}$ as a left $kG$-module and as a right $D_{\lambda}$-module.

\emph{Assume now that $k$ is totally real.} Then each $V_{\lambda}$ admits a  positive definite $G$-invariant inner product
$s_{\lambda}:V_{\lambda}\times V_{\lambda}\to k$ and then the  division algebra $D_\lambda$ comes with a natural anti-involution $\sigma\to \sigma^{\ast}$ characterized  by $s_{\lambda}(v\sigma,v')=s_{\lambda}(v,v'\sigma^{\ast})$.

For a finitely generated $kG$-module $H$ and $\lambda\in\chi(kG)$, the right $D_{\lambda}$-module on $V_\lambda$
structure determines a left $D_{\lambda}$-module structure on $\Hom_{kG}(V_{\lambda},H)$ given as
$(du)(v):=u(vd)$ where $d\in D_{\lambda}$ and $v\in V_{\lambda}$. This makes the natural map
\begin{eqnarray*}
	\oplus_{\lambda\in\chi(kG)}V_{\lambda}\otimes_{D_{\lambda}}\Hom_{kG}(V_{\lambda},H)&\to& H\\
	v\otimes_{D_{\lambda}}u\in V_{\lambda}\otimes_{D_{\lambda}}\Hom_{kG}(V_{\lambda},H)&\mapsto& u(v)
\end{eqnarray*}
an isomorphism of $kG$-modules. This is called the \textit{isotypical decomposition} of $H$ and the image of
$V_{\lambda}\otimes_{D_{\lambda}}\Hom_{kG}(V_{\lambda},H)$ in $H$ is called  the
\emph{isotypical summand} associated to $\lambda$. Any $kG$-linear automorphism of $H$ will
preserve this $G$-isotypical decomposition and acts on each isotypical summand
$V_{\lambda}\otimes_{D_{\lambda}}\Hom_{kG}(V_{\lambda},H)$ through a $D_{\lambda}$-linear transformation on the second tensor factor.
This identifies  $\End_{kG}(H)$ with  $\Pi_{\lambda\in\chi(G)}\End_{D_{\lambda}}(\Hom_{kG}(V_{\lambda},H))$.

Let us suppose further that our $kG$-module $H$ is endowed with a nondegenerate $G$-invariant symplectic form $(a,b)\in H\times H\mapsto \la a,b\ra\in k$ (for example when $H$ is $H_1(C,k)$ for some smooth complex-projective curve $C$ and the symplectic form being the intersection product). {The discussion above shows that  the $G$-centralizer  $\Sp(H)^G$ of $\Sp(H)$ decomposes as}
\begin{equation*}
	\Sp(H)^G=\Pi_{\lambda\in\chi(G)}\Sp(V_{\lambda}\otimes_{D_{\lambda}}\Hom_{kG}(V_{\lambda},H))^G.
\end{equation*}
For a fixed character $\lambda\in\chi(kG)$ and fixed $u,u'\in V_{\lambda}\otimes_{D_{\lambda}}\Hom_{kG}(V_{\lambda},H)$, the map $(v,v')\in V_{\lambda}\times V_{\lambda}\mapsto \la u(v),u'(v')\ra\in k$ is a $G$-invariant bilinear form on $V_{\lambda}$. Since $s_{\lambda}$ is a nondegenerate, $G$-invariant {symmetric bilinear form on $V_{\lambda}$, this implies that there exists a unique $h_{\lambda}(u,u')\in D_{\lambda}$ such that $\la u(v),u'(v')\ra=s_{\lambda}(vh_{\lambda}(u,u'),v')$ for all $v,v'\in V_{\lambda}$.} We then observe that
\begin{eqnarray*}
	h_{\lambda}(d u,u')&=&d h_{\lambda}(u,u'), \ d\in D_{\lambda}\\
	h_{\lambda}(u,u')&=&-h_{\lambda}(u',u)^{\ast}
\end{eqnarray*}
{Hence the pairing $(h,\ast)$ defines a $D_{\lambda}$-valued skew-Hermitian form on the $D_{\lambda}$-module $V_{\lambda}\otimes_{D_{\lambda}}\Hom_{kG}(V_{\lambda},H)$. We thus get an identification}
\begin{equation*}
	\Sp(V_{\lambda}\otimes_{D_{\lambda}}\Hom_{kG}(V_{\lambda},H))^G\cong {U_{D_{\lambda}}(\Hom_{kG}(V_{\lambda},H))}
\end{equation*}
where $U_{D_{\lambda}}(\Hom_{kG}(V_{\lambda},H))$ consists of the $D_\lambda$-linear automorphisms of $\Hom_{kG}(V_{\lambda},H)$ that preserve the  skew-Hermitian form above.

\subsection{Integral Representation}
Suppose $R$ is a domain and $K$ its field of fractions. Given  a finite dimensional (not necessarily commutative) $K$-algebra $A$, then an  \emph{$R$-order} in $A$ is a subalgebra $\Lambda\subset A$ that spans $A$ over $K$ and is as an $R$-submodule is finitely generated. If we are given a left $A$-module $M$ that is of finite dimension over $K$, then an \emph{$\Lambda$-lattice} in $M$ is   finitely generated torsion free $\Lambda$-submodule of $M$ that spans $M$ over $K$. Here are some examples for $G=\Isr$ that will play a special role in this paper.

\begin{eg}
	Consider the case  $R=\Zds$ (so that $K=\Qds$). Then  the integral group ring $\Zds\Isr$ is a $\Zds$-order in $\Qds\Isr$. This is the case that will concern us most.
\end{eg}

\begin{eg}\textbf{(The integral form of $V$)}\label{inforVo}
	As agreed earlier, we regard $\Isr$ as the subgroup of even  permutations of the $5$-element set $\Ksr$.
	The reflection representation $V$ of $\Isr$ is the quotient of $\Cds^\Ksr$ modulo its main diagonal. We get an integral form $V_o$ by taking $\Zds^\Ksr$ instead so that we have a short exact sequence
	\begin{equation}\label{exsqV}
		0\to {\Zds}\to \Zds^{\Ksr}\to V_o\to 0
	\end{equation}
	of $\Zds \Isr$-modules. If we endow $\Zds^\Ksr$ with the natural inner product for which the natural basis (identified with $\Ksr$) is orthonormal, then this identifies $\Zds^\Ksr$ with its  dual and the dual of this exact sequence
	\begin{equation}\label{dexsqV}
		0\to V_o^{\vee}\to \Zds^\Ksr\xrightarrow{\text{sum}} {\Zds}\to 0
	\end{equation}
	is still exact. Here $V_o^{\vee}$ is the set of vectors with coefficient sum zero. This is just the root lattice of type $A_4$ whose roots are differences of distinct basis vectors with  $\Isr$ realized  as the orientation preserving part of its Weyl group.
\end{eg}
\begin{eg}\textbf{(The integral form of $W$)}\label{inforWo}
	The irreducible representation $W$ of $\Isr$ of dimension $5$ has an integral form $W_0$ defined as
	follows. Consider the collection $\Lsr$ of pairs of opposite faces of the dodecahedron (or equivalently, the axes of the order $5$ rotations in $\Isr$). This set has $6$ elements. The group $\Isr$ acts transitively on $\Lsr$,  the stabilizer of each element  being a dihedral group of order $10$. This makes $\Zds^\Lsr$ a $\Zds\Isr$-module which contains the diagonal spanned by $\sum_{l\in \Lsr} l$ as a trivial submodule (we here identify each $l\in \Lsr$ with its characteristic function in $\Zds^\Lsr$). We define $W_o$ to be the quotient,  so that  the following sequence of $\Zds\Isr$-modules is exact
	\begin{equation}\label{exsqW}
		0\to {\Zds}\to \Zds^\Lsr\to W_o\to 0.
	\end{equation}
	We endow $\Zds^\Lsr$ with the $\Isr$-invariant symmetric bilinear form which makes $\Lsr$ an orthonormal base.
	This form identifies $\Zds^\Lsr$ with its dual as a $\Zds\Isr$-modules. So the dual of the exact sequence above is
	\begin{equation}\label{dexsqW}
		0\to W_o^{\vee}\to \Zds^{\Lsr}\xrightarrow{\text{sum}} {\Zds}\to 0,
	\end{equation}
	where $W_o^{\vee}$ is the set of vectors in $\Zds^\Lsr$  whose coefficient sum is zero. Note that $W_o^{\vee}$  is generated by differences of distinct basis vectors; these have self-product $2$ and are the roots of a root system of type $A_5$.
\end{eg}

\begin{eg}
	\textbf{(The integral form of $E$)}
	Let $I_{\Rds}$ be the ambient Euclidean vector space of $D$. We view this as a $\Rds \Isr$-module.  It is irreducible, even its  complexification $I$ is an irreducible $\Cds\Isr$-module, but $\Isr$ is not definable over $\Qds$.
	If $I'$ is obtained from $I$ by precomposing the $\Isr$-action with an outer automorphism of $\Isr$, then $E:=I\oplus I'$  is as a representation is naturally defined over $\Qds$, for a character computation shows that we can take $E_{\Qds}:=\wedge^2 V_{\Qds}$. This representation is even irreducible over $\Qds$, for the splitting requires that we pass to the extension  $\Qds(\sqrt{5})$. Indeed, $E_{\Qds(\sqrt{5})}:=E_{\Qds}\otimes_{\Qds}\Qds(\sqrt{5})$ splits into two 3-dimensional irreducible components that represent $I$ and $I'$
	and its associated division algebra is the field $\Qds(\sqrt{5})$.
	So we can take  $E_o:=\wedge^2 V_o$ as an integral form of $E$. The exact sequence \eqref{exsqV} gives a surjective map $\wedge^2\Zds^\Ksr\to \wedge^2 V_o$ whose kernel is identified with $\Zds^\Ksr\wedge (\sum_{i\in\Ksr}e_i)$, so that we have the  exact sequence of $\Zds\Isr$-modules
	\begin{equation}
		0\to V_o\to \wedge^2 \Zds^\Ksr\to E_o\to 0, \quad E_o:=\wedge^2 V_o.
	\end{equation}
\end{eg}

\begin{prop}\label{onegenr}
	The $\Zds\Isr$-modules $V_o$, $\wedge^2 \Zds^\Ksr$, $E_o=\wedge^2 V_o$, $W_o$ and $\wedge^2 W_o$ are all principal, i.e., generated over $\Zds\Isr$ by one element.
\end{prop}
\begin{proof}
	It is clear that the $\Zds\Isr$-module $V_o$ resp. $W_o$ is generated by any base element of $\Zds^\Ksr$ resp. $\Zds^\Lsr$. The $\Zds\Isr$-module $\wedge^2 \Zds^\Ksr$ is generated by any element of the form $a\wedge b$ where $a$ and $b$ are distinct elements of the standard basis of $\Zds^{\Ksr}$. The proof for $\wedge^2 V_o$ and $\wedge^2 W_o$ is similar.
\end{proof}

It was shown in \cite{zi2021geometry} that for a smooth member $C$ of the Winger pencil, its space  of holomorphic forms $H^0(C,\omega_{C})$ is as a $\Cds \Isr$-module isomorphic to $V\oplus I\oplus I'=V\oplus E$. This implies that $H^1(C;\Cds)$ is isomorphic to $V^{\oplus2}\oplus E^{\oplus 2}$. Since both $V$ and $E$ are complexifications of
irreducible $\Qds\Isr$-modules $V_\Qds$ resp. $E_\Qds$ (which are therefore self-dual), it follows that the canonical isotypical  decomposition for  $H_1(C;\Qds)$ is
\begin{equation}\label{eqn:homologydec}
	H_1(C;\Qds)\cong(V_\Qds\otimes \Hom_{\Qds\Isr}(V_\Qds,H_1(C,\Qds)))\oplus(E_\Qds\otimes \Hom_{\Qds\Isr}(E_\Qds,H_1(C, \Qds))
\end{equation}
with $\dim_\Qds \Hom_{\Qds\Isr}(V_\Qds,H_1(C,\Qds))=2$ and $\dim_{\Qds(\sqrt{5})} \Hom_{\Qds\Isr}(E_\Qds,H_1(C, \Qds))=2$.
We will here focus on the monodromy representation on the first summand.

\section{A Geometric Model of a genus ten curve with icosahedral action}\label{GMg10}
Recall that the Winger penicl is defined as a hypersurface by the following equation in the projective variety $\Pds(I)\times \Bsr\cong \Pds^2\times\Pds^1$
\begin{equation}
	g_2^3+tg_6=0
\end{equation}
Here $t\in \Bsr$ be a parameter, $g_2^3$ and $g_6$ are two generators of $\Cds[I]^{\Isr}_6$ where $g_2$ is a polynomial of degree two representing a smooth conic and $g_6$ is a polynomial of degree 6 representing the union of 6 lines.
In this following section we introduce a geometric model for a smooth fiber $C$ and describe
two stable degenerations in terms of it. We will exploit the fact that the Winger pencil comes with a natural real structure, which
in terms of our modular  interpretation is the map which replaces the given complex structure by the conjugate complex structure (so complex multiplication by  $\sqrt{-1}$ in a tangent space becomes multiplication by $-\sqrt{-1}$).  This indeed defines an anti-holomorphic automorphism of the pencil (acting on both its base and its total space and commuting with the projection). This action is also evident from the explicit form of the pencil (which has real coefficients). In particular, it takes the coordinate $t$ of the base $\Bsr$ for the Winger pencil to $\bar t$. Recall that all the singular members of the pencil appear for real values of $t$:  for $t=0$  we have a triple conic, for  $t=27/5$ an irreducible curve with $6$ nodes,  for $t=\infty$ a union of $6$ lines without triple point and  for $t=-1$ an irreducible curve with $6$ nodes.

Let $\hat{\Sigma}$ be obtained from the dodecahedron $D$ by removing in a $\Isr$-invariant manner a small regular triangle centered at each vertex of $D$ so that the faces of $\hat{\Sigma}$ are oriented solid $10$-gons (in other words it is a truncated regular dodecahedron without triangular faces).
The set of such faces has  12-elements and comes with an antipodal involution. The group permutes these faces transitively and preserves their natural orientations. The boundary of each face consists of two types of edges. We call the ones coming from the edges of $D$ \emph{$1$-cells of edge type}. They are 30 in number. They are not naturally oriented since for every such edge there is rotational symmetry of order two which reverses its orientation. But if it is given as a boundary edge of a face, then it acquires one.

We now identify opposite points on the boundary of $\hat\Sigma$ and thus obtain a {closed, combinatorial  $\Sigma$. {The antipodal involution is orientation reversing on the boundary of $\hat\Sigma$} and this makes that $\Sigma$ is oriented. Since $\Sigma$ has 12 faces, 60 edges and 30 vertices, its  Euler's characteristic is $-18$ a hence the genus is 10.  It comes endowed with an action of $\Isr$ (See Figure \ref{TruD}) which respects the cellular decomposition:} the set of $0$-cells are represented as antipodal pairs of $0$-cells of $\hat\Sigma$ and are naturally indexed by (unordered) antipodal pairs of oriented edges of $D$. The $2$-cells are of course bijectively indexed by the faces of $D$ and are canonically oriented. The $1$-cells come in two types: those that lie on edge of $D$ (hence called of \emph{edge type}) and those that come from the boundary of $\hat\Sigma$ (hence called of \emph{truncation type}).

\begin{figure}
	\centering
	{\includegraphics[width=0.4\textwidth]{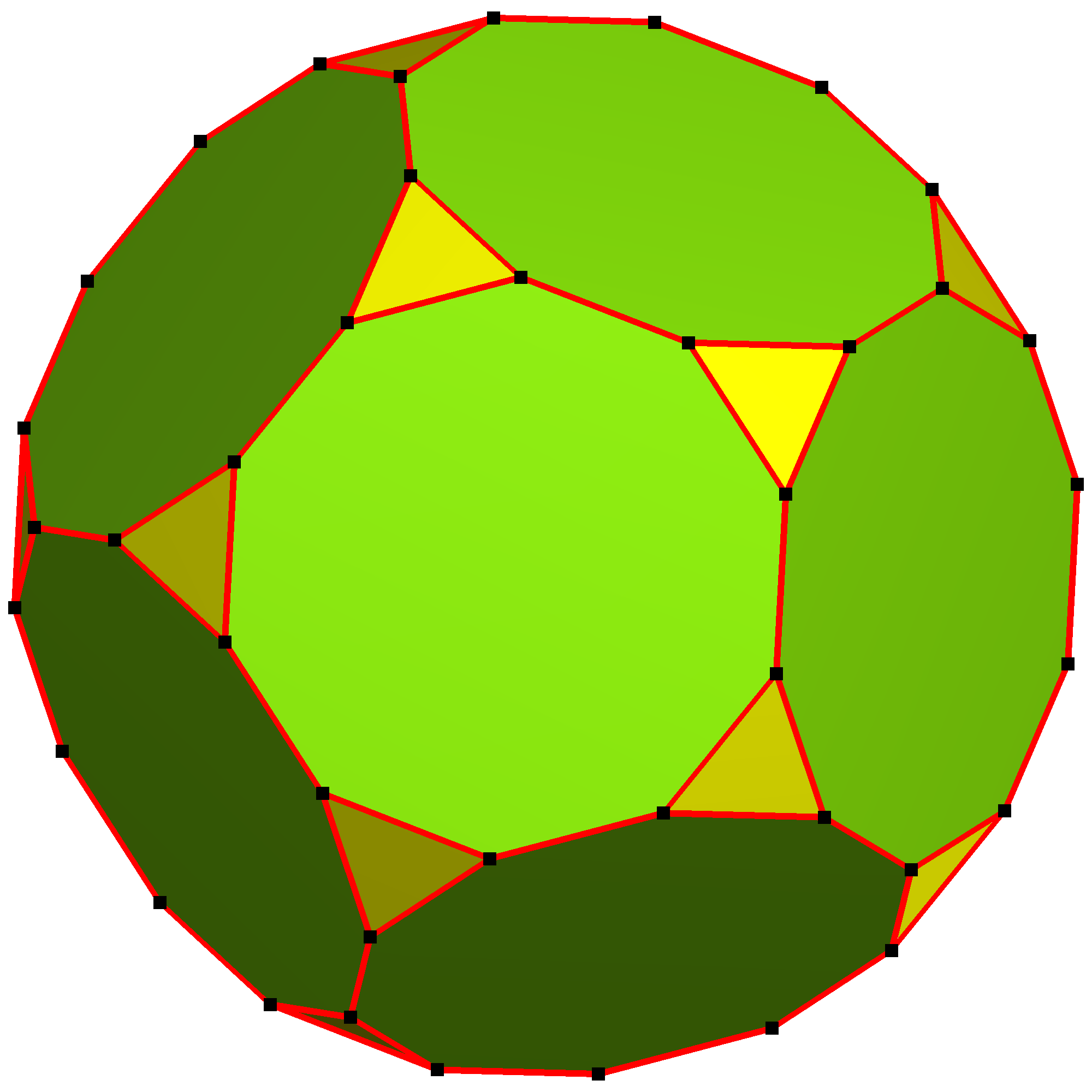}}
	\caption{\small{Removing in a $\Acal_5$-invariant manner a small regular triangle centered at each vertex of dodecahedron and identifying opposite points on the boundary.}}
	\label{TruD}
\end{figure}

The set $\Csr_\edge(\Sigma)$ of \emph{oriented} $1$-cells of $\Sigma$ of edge type is in bijective correspondence with the  the set $\Csr_1(D)$ of oriented edges of the dodecahedron $D$. This bijection is not just $\Isr$-equivariant, but also compatible with orientation reversal.
The set $\Csr_\trc(\Sigma)$ of oriented $1$-cells of $\Sigma$ of truncation type are
also bijectively indexed by $\Csr_1(D)$, but here orientation reversal is induced by the antipodal map. The following proposition is now clear.


\begin{prop}\label{sporIp}
	The action of $\Isr$ on the cells of $\Sigma$ is as follows:
	\begin{enumerate}
		\item the action of $\Isr$ on the set $\Csr_0(\Sigma)$ of  $0$-cells of $\Sigma$ is transitive, each $0$-cell having a stabilizer cyclic of order $2$,
		\item the set $\Csr_1(\Sigma)$ of oriented $1$-cells of $\Sigma$  consists of two regular $\Isr$-orbits $\Csr_\trc(\Sigma)$ and $\Csr_\edge(\Sigma)$,
		\item the action of $\Isr$ on the set $\Csr_2^+(\Sigma)$ of canonically oriented $2$-cells is transitive,
		      the stabilizer each such cell being cyclic of order $5$.
	\end{enumerate}
\end{prop}

An  oriented  $1$-cell of $\Sigma$ is part of a unique loop  consisting of oriented cells of the same type. Let us analyse this in some detail.

A \emph{loop of truncation type} consists of three oriented $1$-cells of that type and each  oriented $1$-cell of truncation type
appears in a unique such loop. They are bijectively indexed by the set $\Csr_0(D)$ of vertices
of $D$: every vertex $x$ is at the center of a solid triangle whose interior has been removed to form
$\Sigma$ and the boundary of this triangle with its counterclockwise orientation is a sum $\delta_x$ of
three oriented $1$-cells of truncation type. We have $\delta_{\iota x}=-\delta_x$.
We will call the closed loops constructed in this way loops \emph{of truncation type}.
We have 20 such closed
loops (10 if we ignore orientation) and the $\Isr$-action permutes them transitively. Hence the $\Isr$-stabilizer of
one such closed loop is cyclic of order $3$. We will denote this set of twenty 1-cycles by $\Delta_{\trc}$.

A \emph{loop of edge type} is the sum of oriented $1$-cell of that type plus its image under $-\iota$. These are bijectively indexed by the set $\Csr_1(D)$ with orientation reversal induced by the antipodal map. We denote this labeling  $y\in \Csr_1(D) \mapsto \delta_y$. Note that then  $\delta_{\iota y}=-\delta_y$ and $\delta_{-y}=-\delta_y$. The set of such $1$-cycles, that we shall denote  by $\Delta_{\edge}$, is an $\Isr$-orbit of 30 elements (the $\Isr$-stabilizer of one such loop is of order two).

\begin{rmk}\label{rmk:basis}
	If we remove the loops of truncation type, then the result is the interior of $\tilde\Sigma$, which is topologically a sphere with $20$ punctures. This implies that their classes $[\delta_y]\in H_1(\Sigma)$ span a  sublattice $L_\trc\subset H_1(\Sigma)$ that is Lagrangian with respect to the intersection pairing: if we select a system of representatives  $R\subset \Csr_0(D)$ for the action of the antipodal involution acting on $\Csr_0(D)$, then the ten element set $\{[\delta_x]\}_{x\in R}$ is a basis for $L_\trc$ and spans a maximal isotropic subgroup of $H_1(\Sigma)$. We can extend this to a basis of $H_1(\Sigma)$ as follows: choose for each $x\in R$ a path $\tilde\gamma_x$ on $\tilde \Sigma$ from a point $x'$ of the component of $\partial\tilde \Sigma$ which has $x$ as its center to its antipode $\iota x'$ on $\tilde\Sigma$. Then  the image $\gamma_x$ of $\tilde\gamma_x$ in $\Sigma$ is a loop whose homology class $[\gamma_x]\in H_1(\Sigma)$ has the property that $\la [\gamma_x], \delta_{x'}\ra$ is zero for $x'\in R$ unless $x'=x$ in which case it is $1$. If we let $x$ run over $R$, then the twenty 1-cycles $\delta_x$ and $\gamma_x$ map to a basis of $H_1(\Sigma)$ (which need not be symplectic).
\end{rmk}

The following is straightforward to check.

\begin{lem}\label{lemma:}
	The intersection numbers of these $1$-cycles are as follows: any two loops of the same type have intersection number zero and if $x\in \Csr_0(D)$ and $y\in \Csr_1(D)$, then $\la \delta_x, \delta_y\ra =0$ unless $x$ lies on $y$ or on $\iota y$, in which case $\la \delta_x, \delta_y)\ra\in\{\pm 1\}$ with the plus sign appearing if and only if $x$ is the end point of $y$ or the initial point of $\iota x$.
\end{lem}

\subsection{Degenerations of $\Sigma$}
We will describe two  degenerations of the combinatorial genus ten surface $\Sigma$ with  $\Isr$-action that have $\Delta_\trc$ resp.\ $\Delta_\edge$ as their set of vanishing cycles.

We begin with giving a one-parameter family  piecewise Euclidean structures on $\hat{\Sigma}$. For this we assume that the length of a $1$-cell of edge type  of $D$ is $\tau>0 $ and  the length of an
$1$-cell of truncation type is $1-\tau>0$. It is then clear that this determines $\hat{\Sigma}$  as a metric space, the metric being
piecewise Euclidean and invariant under  both $\Isr$ and the antipodal
involution. Such a metric defines a conformal structure $J_\tau$, a priori only defined on $\Sigma$  minus its vertices, but one that is well-known to extend across them. The given orientation makes this then a $\Isr$-invariant complex structure. Taking the opposite orientation will give us the complex conjugate complex structure $-J_\tau$. This shows that $(\Sigma,J_\tau)$ is defined over $\Rds$.

If we let $\tau$ tend to $1$, then  the length of closed loop of truncation type tends to $0$ and we get a piecewise flat metric on the singular surface $\Sigma_\trc$ that is obtained from $\Sigma$ by contracting each truncation cycle to a point.
Note that this singular surface is also got by identifying the opposite vertices of the regular dodecahedron $D$. The metric makes this a singular irreducible $\Isr$-curve with $6$ nodes, isomorphic with $C_{\frac{27}{5}}$.
Similarly, if we let $\tau$ tend to $0$, the length of closed loop of edge type tends to $0$ and we get a piecewise flat metric on the singular surface $\Sigma_\edge$ that is obtained from $\Sigma$ by contracting each edge type cycle to a point. In this case $\Sigma_\edge$  minus its singular points consists of six $5$-punctured spheres (each obtained by glueing two regular pentagons along their boundary and subsequently removing the vertices) which with the complex structure  becomes  isomorphic to $C_{\infty}$, the union of 6 lines.

We sum this up with  the following proposition.

\begin{prop}
	The Riemann surface $(\Sigma, J_\tau)$ is the set of complex points of a complex real algebraic curve. It has genus 10 and comes with a faithful $\Isr$-action, hence is isomorphic to a member of the Winger pencil.
	We thus have defined a continuous map $\gamma: [0,1]\to \Bsr$ which traverses the real interval  $[\infty, \frac{27}{5}]$
	and  which maps $(0,1)$ to $\Bsr^\circ$ (and so lands in the locus where $t$ is real and $>\frac{27}{5}$),
	such that the pull-back of the Winger pencil  yields the family constructed above. The degenerations of $\Sigma$ into $\Sigma_\edge$ resp. $\Sigma_\trc$ have $\Delta_\edge$ resp.\ $\Delta_\trc$ as their sets of vanishing cycles.
\end{prop}

We are not claiming here that $\gamma$ is a homeomorphism onto its image, although that is likely to be true (we expect the derivative of $\gamma$ to be nonzero on $(0,1)$, where it is indeed differentiable).

\subsection{Cellular Homology}\label{CellHo}
The cellular decomposition of  $\Sigma$ enables us compute its homology as that of the combinatorial chain complex
\begin{equation*}
	\xymatrix{
		0\ar[r]&C_2\ar[r]^{\partial_2} &C_1\ar[r]^{\partial_1}& C_0\ar[r]& 0.
	}
\end{equation*}
This is a complex of $\Zds\Isr$-modules with the middle term decomposing as $C_\edge\oplus C_\trc$.
In particular, $Z_i:=\ker(\partial_i)$ by $Z_i$ and  $B_{i-1}:=\Img(\partial_i)$ are also $\Zds\Isr$-modules.
Lemma \ref{sporIp} tells us what they are: let $z$ be a face of the dodecahedron $D$ (hence canonically oriented)
and let  $y\in\Csr_1(D)$ be an edge of $z$ endowed with its counterclockwise orientation. We then may note here that the group $\Isr$ acts simply  transitively on such pairs $y<z$. It is clear that $z$ determines an oriented face of $\Sigma$ and hence an element of
$C_2$. As we have seen, $y$ determines an element of $\Csr_\edge$ and an element of $\Csr_\trc$.
Then
\begin{enumerate}
	\item $C_0\cong \Zds\Isr /(h_0-1)\Zds\Isr$ with a generator represented by $y$ and $h_0\in \Isr$ sending $y$ to $\iota y$,
	\item $C_\trc\cong \Zds\Isr /(h_0+1)\Zds\Isr$ with  $y$ and $h_0$ as above,
	\item $C_\edge\cong \Zds\Isr /(h_1+1)\Zds\Isr$ with a generator represented by $y$ and $h_1\in\Isr$ sending $y$ to $-y$,
	\item  $C_2\cong \Zds\Isr /(h_2-1)\Zds\Isr$ with a generator represented by $z$ and
	      $h_2\in \Isr$ inducing a counter clockwise rotation over $2\pi/5$ in $z$
\end{enumerate}

If we apply the left exact functor $\Hom_{\Zds\Isr} (V_o, -)$ to the exact sequence
\begin{equation}\label{seqph1}
	\xymatrix{
	0\ar[r]&B_1\ar[r] &Z_1\ar[r]^-{p}& H
	_1(\Sigma)\ar[r]& 0
	}
\end{equation}
then get the exact sequence
\begin{equation}\label{eqn:bacicexseq}
	\xymatrix@C=1pc{0\ar[r]&\Hom_{ \Zds\Isr}(V_o,B_1)\ar[r]&\Hom_{ \Zds\Isr}(V_o,Z_1)\ar[r]^-{p_*}&\Hom_{ \Zds\Isr}(V_o,H_1(\Sigma))}.
\end{equation}
We will now define two elements of $\Hom_{ \Zds\Isr}(V_o,Z_1)$:  one that takes values in $Z_\trc$ and is
denoted $u_\trc$ and another taking values in $Z_\edge$ and is denoted $u_\edge$.
We will subsequently prove that they generate $\Hom_{ \Zds\Isr}(V_o,Z_1)$.

Recall that the dodecahedron $D$ has five inscribed cubes, meaning that the eight vertices of such a cube are
also vertices of the dodecahedron  and that we denoted the set of such cubes by $\Ksr$.
Every vertex of $D$ appears in exactly two inscribed cubes.

Let us fix one such a cube $e\in \Ksr$ (as in Figure \eqref{cid}). The corresponding generator of $\Zds^\Ksr$ (which we also denote by $e$) has an image in $V_o$ that we shall denote by $\bar{e}$. The set  of (eight) vertices of $e$ decomposes into two disjoint $4$-element
subsets $E$ and its antipode $\iota E$ such that  no two elements of $E$ span an edge of the cube.
Both $E$ and $\iota E$ are orbits of the $\Isr$-stabilizer of the cube.
The $\Isr$-stabilizer of $E$ is of order $12$ (it is the group of even permutations of $E$) so that $E$ has
exactly five $\Isr$-translates. We put
\[
	\delta^E_\trc:=\sum_{x\in E} \delta_x
\]
This is a $4$ term sum of elements of $\Delta_\trc$ and a $12$ term sum of elements of $\Csr_\trc$.
The $\Isr$-orbit of $\delta^E_\trc$ has $5$ elements.
Hence the sum of the elements of this orbit has 60 terms with each oriented cell of truncation type appearing and so this sum must be zero.
It follows that $\delta^E_\trc$ generates a $\Zds\Isr$-submodule of type $V_o$ and hence defines
an equivariant homomorphism $u_\trc: V_o\to Z_{\trc}$ with $u_\trc(\bar{e})=\delta^E_\trc$.

Each vertex $x$ of the dodecahedron $D$ determines $3$ oriented edges of $D$ (namely those
that have $x$ as initial point) and each such oriented edge defines an element of $\Delta_\edge$ (a loop on $\Sigma$ of edge type).  We take the sum of these three and then also sum over $E$  and denote the resulting sum of
$12$ elements of $\Delta_\edge$  by $\delta_\edge^E$, so
\[
	\delta^E_\edge:=\sum_{x\in E} \sum_{\{y\in \Csr_1(D): in(y)=x\}} \delta_y
	=\sum_{x\in E} \sum_{\{y\in \Csr_1(\Sigma): in(y)=x\}} y-\sum_{x\in \iota E} \sum_{\{y\in \Csr_1(\Sigma) : in(y)=x\}} y,
\]
where $in(y)$ stands for the initial point of the oriented edge $y$.
It is clear that the $\Isr$-orbit of $\delta_\edge^E$ has size $5$.
Hence the sum of the elements of this orbit has 120 terms with each oriented cell of edge type appearing twice
and so must be zero. Then the $\Zds\Isr$-submodule
generated by $\delta_\edge^E$ is of type $V_o$ so that we get an equivariant homomorphism
$u_{\edge}: V_o\to Z_{\edge}$  with $u_\edge(\bar{e})=\delta_\edge^E$.

The main result of this section is:

\begin{prop}\label{propcwseq}
	In the exact sequence \eqref{eqn:bacicexseq} the term $\Hom_{ \Zds\Isr}(V_o,B_1)$ is trivial, $\Hom_{ \Zds\Isr}(V_o,Z_1)$ is the free abelian group generated by $u_\trc$ and $u_\edge$  and the cokernel of $p_*$ is cyclic of order $3$, to be precise, the image of $u_\edge$ in $\Hom_{ \Zds\Isr}(V_o,H_1(\Sigma))$ is divisible by $3$,  so that
	$\Hom_{ \Zds\Isr}(V_o,H_1(\Sigma,\Zds))$ is the free abelian group generated by the images $u_\trc$ and $\frac{1}{3}u_\edge$.
\end{prop}

\begin{rmk}
	We can phrase this in terms of a topological Jacobians as follows. By definition the topological Jacobian $J(\Sigma)$ of $\Sigma$ is the real torus $H_1(\Sigma,\Rds/\Zds)$. The name is justified, because as is well known, a complex structure on $\Sigma$ puts such a structure on $J(\Sigma)$ for which it then
	becomes the  Jacobian of the resulting Riemann surface. The above proposition suggests that we consider the homomorphism of real 2-dimensional tori
	\[
		\tilde J(V_o, \Sigma):=\Hom_{ \Zds\Isr}(V_o, Z_1\otimes \Rds/\Zds)\to \Hom_{ \Zds\Isr}(V_o,H_1(\Sigma,\Rds/\Zds))=: J(V_o, \Sigma).
	\]
	Proposition \ref{propcwseq} tells us that this is a covering of degree $3$ whose kernel is generated by the image   of
	$\frac{1}{3}u_\edge$ in $\tilde J(V_o, \Sigma)$ (a point of order $3$). We shall see that the natural homomorphism    $V_o\otimes \Hom_{ \Zds\Isr}(V_o,H_1(\Sigma))\to H_1(\Sigma)$ has torsion free cokernel. This implies that the natural map of real  tori $V_o\otimes J(V_o, \Sigma)\to J(\Sigma)$ is injective with image  the torus
	defined by the reflection representation $V$. A $\Isr$-invariant complex structure on $\Sigma$ turns into $J(\Sigma)$ a  Jacobian on which
	$\Isr$-acts. Thus $V_o\otimes J(V_o, \Sigma)$ inherits a $\Isr$-invariant complex structure, so that we also get a complex structure on $J(V_o, \Sigma)$, making it an elliptic curve. Since  $\tilde J(V_o, \Sigma)\to J(V_o, \Sigma)$ is a covering of degree $3$, this makes  $\tilde J(V_o, \Sigma)$ an an elliptic curve that comes with a point of order $3$.
\end{rmk}

\begin{figure*}
	\centering
	{\includegraphics[width=0.48\textwidth]{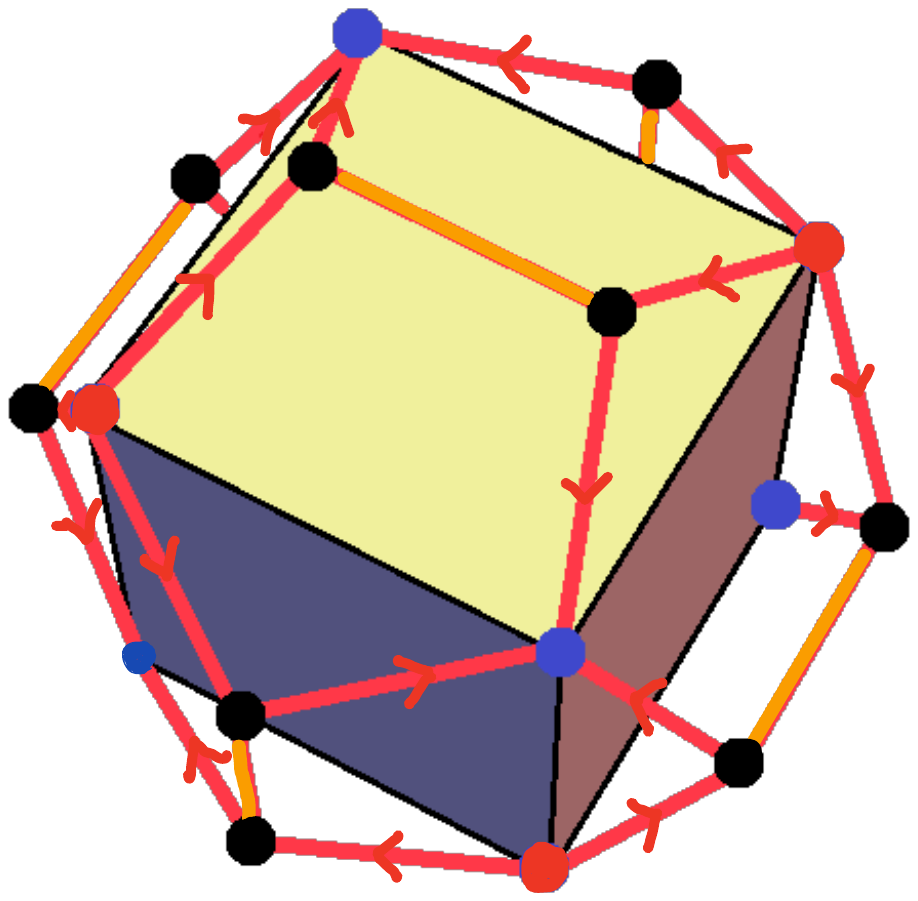}}
	\caption{\small{Red vertices belong to $E$ and are sources; blue vertices belong to $\iota E$ and are sinks.}}
	\label{cid}
\end{figure*}

Before we prove the Proposition \ref{propcwseq}, we list the intersection numbers of $u_{\edge{}}(\bar{e})$ resp.\ $u_{\trc}(\bar{e})$ with the elements of $\Delta_{\edge}$ and $\Delta_{\trc}$. Knowing these numbers will also be important for computing the local monodromies.

\begin{lem}\label{speint}
	Let $E$, $u_{\edge}$ and $u_{\trc}$ be as defined above, $\Delta_{\edge}$ and  $\Delta_{\trc}$ as defined in the last section. Then the class $[u_\edge(\bar{e})]$ resp.\  $[u_\trc(\bar{e})]$ has zero intersection number with the elements of $\Delta_\edge$ resp.\ $\Delta_\trc$, whereas
	for  $x\in \Csr_0(D)$ resp. $y\in \Csr_1(D)$,
	\[
		\la [u_{\edge}(\bar{e})],[\delta_x]\ra=
		\begin{cases} 3   & \text{if }  x\in E,       \\
              -3, & \text{if }  x\in \iota E, \\
              0   & \text{otherwise.}
		\end{cases}
		\quad
		\la [u_{\trc}(\bar{e})],[\delta_y]\ra=
		\begin{cases} -1 & \text{if }  in(y)\in E,       \\
              1, & \text{if }  in(y)\in \iota E, \\
              0  & \text{otherwise.}
		\end{cases}
	\]
\end{lem}
\begin{proof}
	This is clear from the definitions (see also Figure \ref{cid}).
\end{proof}

In order to show that  $[u_{\edge}(\bar{e})]\in H_1(\Sigma)$ is divisible by $3$ we need:

\begin{lem}\label{lem:mod3}
	The corresponding 1-chain on $D$
	\[
		\tilde\delta_\edge^E:=\sum_{x\in E} \sum_{\{y\in \Csr_1(D): in(y)=x\}} y-\sum_{x\in \iota E} \sum_{\{y\in \Csr_1(D): in(y)=x\}} y
	\]
	is a boundary modulo $3$. In fact, slightly more is true:  if we are given a face $z_o$ of $D$ and write $D(z_o)$ for the
	complement of the interior of $z_o\cup \iota z_o$ in $D$, then $\tilde\delta_\edge^E$ is still a boundary modulo $3$ when regarded as a chain on $D(z_o)$.
\end{lem}
\begin{proof}
	It follows from the definition of $\tilde\delta_\edge^E$ that its boundary is equal to $3\iota E-3E$ (viewed as an element of $C_0(D)$) and  so $\tilde\delta_\edge^E$ is a cycle modulo $3$. Since $H_1(D; \Zds/3)=0$, this must be a boundary modulo $3$. This already implies that
	$\tilde\delta_\edge^E\equiv \sum_{z\in \Csr_2(D)} n_z \partial z \pmod{3 Z_1(D)}$ for certain $n_z\in \Zds$.
	Since $\iota_*$ takes $\tilde\delta_\edge^E$ to $\tilde\delta_\edge^{\iota E}=-\tilde\delta_\edge^E$  and $\iota_*z=-\iota z$,
	we can arrange that opposite faces
	have equal coefficients: just replace $n_z$ by an integer $n'_z$ satisfying $2n'_z\equiv n_z+n_{\iota z}\pmod{3}$.
	In particular, $n_{z_o}=n_{\iota z_o}$. Since the sum of all the (naturally oriented) faces of $D$ has zero boundary, we are free to subtract
	$n_{z_0}$ times that sum. This will make the coefficients of $z_o$ and $\iota z_o$ zero. So $\tilde\delta_\edge^E$ is a boundary modulo $3$ on $D(z_o)$.
\end{proof}

\begin{cor}\label{uedgemod3}
	The image of $u_{\edge}(\bar{e})$ in $H_1(\Sigma)$ is divisible by $3$.
\end{cor}
\begin{proof}
	Let $\cup_{x\in R} \{[\delta_x],[\gamma_x]\}$ be a basis of $H_1(\Sigma)$ as in Remark \ref{rmk:basis}.
	Since the intersection pairing is unimodular, it suffices to show that the intersection number of
	$[(u_{\edge}(\bar{e}))]$ which every basis element is divisible by $3$. For the $[\delta_x]$ this follows from
	Lemma \ref{speint} above. We check this for $\gamma_x$, that is, we show that the intersection number of $\gamma_x$ with
	$\delta^E_\edge$ is divisible by $3$. Recall that $\gamma_x$ is the image of a path
	$\tilde\gamma_x$ on $\tilde\Sigma$ that connects an antipodal pair of points with the initial point
	$\tilde\gamma_x(0)$
	on the boundary component of $\tilde\Sigma$ whose center is $x$. Choose a face $w$ of $D$ which contains
	$x$, but such that the corresponding face $\tilde w$ of $\tilde\Sigma$ does not  $\tilde\gamma_x(0)$. We can then
	arrange that $\tilde\gamma_x$ avoids  the interior of $\tilde w\cup\iota\tilde w$ so that the image of $\tilde\gamma_x$ in $D$ lies in $D(w)$. We can also arrange that $\tilde\gamma_x$ is in general position with respect to the cellular decomposition of $\tilde\Sigma$ in the sense
	that it avoids the vertices and is transversal to the edges.
	The intersection number
	$\la[\tilde\gamma_x],[\delta^E_\edge]\ra$  can then be computed on  $D(w)$: we need to sum over intersection numbers of the image of $\tilde\gamma_x$ in $D(w)$ with $\tilde\delta^E_\edge$ (regarded as chain on $D(w)$).
	But by Lemma \ref{lem:mod3}, $\tilde\delta_\edge^E$ is on $D(w)$ a boundary modulo $3$ and hence
	$\la[\tilde\gamma_x],[\delta^E_\edge]\ra$ will be divisible by $3$.
\end{proof}

\begin{proof}[Proof of the Proposition \ref{propcwseq}]
	The assertion that $\Hom_{ \Zds\Isr}(V_o,B_1)$ is trivial follows if we show that $V$ does not appear in
	$C_2\otimes \Cds$. The latter is the complexified permutation representation of $\Isr$ on the set of faces of
	$D$ whose character is found to be as in the table \eqref{chtbimpal2}. This shows that
	the representation is isomorphic to $\Cds\oplus W\oplus I\oplus I'$, where $W$ is the $5$-dimensional irreducible representation. In particular, $V_o$ does not occur in $B_1$. So by the exact sequence \eqref{eqn:bacicexseq}, the natural map  $\Hom_{ \Zds\Isr}(V_o,Z_1)\to
		\Hom_{ \Zds\Isr}(V_o,H_1(\Sigma))$ is injective. Since we know that $\Hom_{ \Zds\Isr}(V,H_1(\Sigma; \Cds))$ is of dimension $2$, it follows that $\Hom_{ \Zds\Isr}(V_o,H_1(\Sigma))$ has rank $2$.
	\begin{table}
		\centering
		\caption{Character Table of $\Cds\Isr[c_2]$}\label{chtbimpal2}
		\begin{tabular}{cccccc}
			\noalign{\smallskip}\hline\noalign{\smallskip}
			Conjugacy Class & (1) & (12)(34) & (123) & (12345) & (12354) \\
			\noalign{\smallskip}\hline
			                & 12  & 0        & 0     & 2       & 2       \\
			\noalign{\smallskip}\hline
		\end{tabular}
	\end{table}

	It is a priori clear that
	$\Hom_{ \Zds\Isr}(V_o,Z_\trc)\oplus \Hom_{ \Zds\Isr}(V_o,Z_\edge)\subset \Hom_{ \Zds\Isr}(V_o,Z_1)$. Our construction makes it plain that these summands are generated by $u_\trc$ resp.\ $u_\edge$.
	By Lemma \ref{uedgemod3}, the image $[u_\edge]$ of $u_\edge$ in $\Hom(V_o, H_1(\Sigma))$ is divisible by $3$.
	On the other hand, $[u_\edge]$  and $[u_\trc]$ will span $\Hom(V, H_1(\Sigma; \Cds))$ over $\Cds$ and so
	any element of $\Hom(V_o, H_1(\Sigma;))$ is  of the form $a[u_\edge(\bar{e})]+b[u_\trc(\bar{e})]$ for certain constants $a,b\in\Cds$. Lemma \ref{speint} shows that $[u_\edge(\bar{e})]$ resp.\ $[u_\trc(\bar{e})]$ have intersection product $3$ resp.\ $1$ with some homology class and hence $a\in \frac{1}{3}\Zds$ and $b\in\Zds$. It follows that
	$\Hom(V_o, H_1(\Sigma))$ is freely generated by $\frac{1}{3}[u_\edge]$ and $[u_\trc]$
\end{proof}

\section{The Local Monodromy}
Recall that on $\Sigma$ we defined a family of complex structures $J_\tau$ with $\tau\in (0,1)$ which defined a path $\gamma: (0,1)\to \Bsr^\circ$ in the base of the Winger pencil  traversing the positive interval $(\infty, \frac{27}{5})$. This path had a continuous extension to $[0,1]$ that gave rise to the stable degenerations $\Sigma_\edge$ (for $\gamma(0)=\infty$ and $\Sigma_\trc$ (for $\gamma(1)=\frac{27}{5})$). We will determine the monodromies of these degenerations.  For this it is convenient to regard $\gamma|(0,1)$ as a base point for $\Bsr^\circ$ (we here recall that if $X$ is a space, then  any map from a contractible space to $X$ can serve as its base point) and denote the fundamental group of $\Bsr^\circ$ with this base point by $\pi$. So this will then be part of the monodromy representation of $\pi$ on $H_1(\Sigma)$.

If we replace the curve $C$ in \eqref{eqn:homologydec} by $\Sigma$ we obtain  a canonical decomposition
\begin{equation}\label{candecom}
	H_1(\Sigma;\Qds)\cong(V_\Qds
	\otimes \Hom_{\Qds\Isr}(V_\Qds,H_1(\Sigma,\Qds))\oplus(E_\Qds\otimes \Hom_{\Qds\Isr}(E_\Qds,H_1(\Sigma, \Qds))
\end{equation}
Since the monodromy action will preserve this decomposition, we have a monodromy representation
of $\pi$ on both $\Hom_{\Qds\Isr}(V_\Qds,H_1(\Sigma;\Qds))$ and $\Hom_{\Qds\Isr}(E_\Qds,H_1(\Sigma; \Qds))$.
We will focus on the first type and in particular on an integral version of it, namely $\Hom_{\Zds\Isr}(V_o,H_1(\Sigma))$.
We will denote that representation simply by $\rho_{V_o}$. In Subsection \ref{subsect:isotypical}  we observed
(in a much more general setting) that the symplectic form on $H_1(\Sigma; \Qds)$  and the inner product on $V_\Qds$ give rise to a symplectic form on $\Hom_{\Qds\Isr}(V_\Qds, H_1(\Sigma; \Qds))$. Since this space is of dimension two and the inner product on $V_\Qds$ is unique up to a positive scalar, such a form determines a little more than an orientation. Indeed, by Proposition \ref{propcwseq}
\[
	U_\edge:= \tfrac{1}{3}[u_\edge], \quad U_\trc:=[u_\trc]
\]
is a basis of $\Hom_{\Zds \Isr}(V_o,H_1(\Sigma))$ and we can do a rescaling such that $U_\edge\cdot U_\trc=1$. So then $\rho_{V_o}$ takes its values in $\Sp_1(\Zds)\cong\SL(2, \Zds)$.

In the following section we determine $\rho_{V_o}$  for the degenerations $\Sigma_\trc$  and
and $\Sigma_\edge$  and do a local discussion for the other degenerations. In the subsequent section we give a complete description of $\rho_{V_o}$.

If $C_s$ is (singular)  member of the Winger pencil and $U\subset \Bsr$ a small
disk-like neighborhood of $s$ (so that $C_s\subset \Wsr_{U}$ is a homotopy equivalence), then
for  any $t\in U-\{s\}$ the natural map $H_1(C_t)\to H_1(\Wsr_{U})\cong H_1(C_s)$ is onto. So if $L$ denotes the kernel, then we
get the short exact sequence
\begin{equation}\label{kHH}
	0\to L\to H_1(C_t)\to H_1(C_s)\to 0
\end{equation}
In case $C_s$ has only nodal singularities, $L$ is an $\Isr$-invariant isotropic sublattice.

\subsection{The Monodromies of the Degenerations of $\Sigma$}\label{LmSig}
In this section, we will determine the local monodromies  at the end points of $\gamma$.

Let us denote the dual intersection graph of $\Sigma_\edge$  by $G_\edge$. There is a  natural homotopy class of maps $\Sigma_\edge\to G_\edge$  which induces
an isomorphism $H_1(\Sigma_\edge)\to H_{1}(G_\edge)$.
Recall that  $H_1(G_\edge)$ is free of rank 10, so that the kernel
$L_\edge$ of $H_1(\Sigma)\to H_1(\Sigma_\edge)$ is in fact a primitive Lagrangian sublattice. The intersection product then identifies $L_\edge$ with the dual of $H_1(G_\edge)$ so that we get the short exact sequence
\begin{equation}\label{0HHH0}
	\xymatrix{
	0\ar[r]& L_\edge\ar[r]& H_1(\Sigma)\ar[r]^-{\phi}& L_\edge^{\vee}\ar[r]& 0.
	}
\end{equation}

The monodromy transformation $\rho_{\edge}: H_1(\Sigma)\to H_1(\Sigma)$ preserves the exact sequence \eqref{0HHH0} and acts non-trivially only on the middle term. It is given by the Picard-Lefschetz formula:
\begin{equation}\label{PLformu}
	\rho_{\edge}(h)=h+\sum_{l\in \triangle_\edge/\{\pm1\}}\la h,l\ra l,
\end{equation}
where $\triangle_\edge$ denotes the set of vanishing cycles in $H_1(\Sigma)$ defined by the degeneration (this set is invariant under multiplication with $-1$).

Likewise at the other end: if $G_\trc $ is the dual intersection graph of $\Sigma_\trc$, then the kernel of $H_1(\Sigma)\to H_1(\Sigma_\trc)\cong H_1(G_\trc)$ is the primitive Lagrangian sublattice $L_\trc$ we introduced earlier and we get a similar short exact sequence and a similar description of the associated monodromy $\rho_\trc$ in terms of $\Delta_\trc$.

Theorems \ref{monoSig} and Theorem \ref{spltSig} will give the local monodromy in each case and give an interesting property of exact sequence \eqref{0HHH0}. By  Proposition \ref{propcwseq}, the $\Zds$-module $\Hom_{\Zds\Isr}(V_o,H_1(\Sigma))$ is freely generated by  $U_{\edge}$ and $U_{\trc}$ and so it is natural to express the monodromies $\rho_\edge$ and $\rho_\trc$ in terms of these generators.

\begin{thm}\label{monoSig}
	The monodromy $\rho_\trc$ fixes $U_{\trc}$ and takes $U_{\edge}$ to $U_{\edge}+U_{\trc}$ and
	and the monodromy $\rho_\edge$ fixes $U_{\edge}$ and takes $U_{\trc}$ to $U_{\trc}-3U_{\edge}$.
\end{thm}

\begin{cor}\label{cor:monoSIg}
	Let $\alpha$ be a counterclockwise loop in $\Bsr^\circ$ which only contains the punctures  $\infty$ and $\frac{27}{5}$ in its interior and is based at some point on the image of $\gamma$, so that it goes first around $\infty$ and then around $\frac{27}{5}$.
	If $[\alpha]$ denotes its image in $\pi$, then $\rho_{V_o}[\alpha]=\rho_\edge\rho_\trc$ takes  $U_{\trc}$ to $-2U_\edge+U_\trc$ and $U_\trc$ to $-3U_\edge +U_{\trc}$ and hence is of order $3$.
\end{cor}

Before we prove this theorem, we establish some properties of the Lagrangian lattices.

Recall that $\Ksr$ is a five element-set of inscribed cubes of $D$ so that $\Isr$ becomes the group
of even permutations of $\Ksr$. Any vertex of $D$ is contained in precisely two such cubes, and the orientation orders this pair (the opposite vertex gives the oppositely ordered pair). In this this way we produce a
$\Zds\Isr$-linear map $L_\trc\to \wedge^2\Zds^\Ksr$. This is in fact an isomorphism because every ordered pair
of distinct inscribed cubes is associated to a vertex. We will therefore identify these two $\Zds\Isr$-modules.
Note that under this isomorphism $\Delta_\trc$ is identified with the collection of 20 vectors $e\wedge e'$,
with $(e,e')$ distinct elements of $\Ksr$.

The short exact sequence \eqref{kHH} becomes the following short exact sequence of $\Zds\Isr$-modules
\begin{equation}\label{LHLv}
	\xymatrix{
	0\ar[r]& L_\trc\ar[r]\ar[dr]_{j_\trc}& {H_1(\Sigma,\Zds)}\ar[r]^-{\phi}& L_\trc^{\vee}\ar[r]& 0\\
	&&{Z_1}\ar[u]&&
	}
\end{equation}
Here $j_\trc$ is the obvious map. Let us apply the left exact functor $\Hom_{\Zds\Isr}(V_o,\cdot)$ to the short exact sequence \eqref{LHLv} and combine it with the exact sequence \eqref{seqph1}
\begin{equation}\label{hom275}
	\xymatrix@C=1pc@R=2pc{
	0\ar[r]& \Hom_{\Zds\Isr}(V_o,L_\trc)\ar[r]\ar[dr]_{j_{\trc *}}&\Hom_{\Zds\Isr}(V_o,H_1(\Sigma))\ar[r]^-{\phi_*}&\Hom_{\Zds\Isr}(V_o, L_\trc^{\vee})\ar[r]& \Ext_{\Zds\Isr}(V_o,L_\trc)\\
	&&\Hom_{\Zds\Isr}(V_o,Z_1)\ar[u]&&
	}
\end{equation}
By Proposition \ref{propcwseq}, the vertical arrow $\Hom_{\Zds\Isr}(V_o,Z_1)\to \Hom_{\Zds\Isr}(V_o,H_1(\Sigma))$ is injective.
\\

The dual intersection graph $G_\edge$ has six vertices and every two vertices are joined by an edge. Hence in this case we get the complete graph with six vertices, i.e., a graph of type $K_6$. Let $n:\hat\Sigma_\edge\to \Sigma_\edge$ be the normalization map. 

The set  of connected components of $\hat\Sigma_\edge$ has $6$ elements and $\Isr$ acts on it by permutations. There is a (unique) $\Isr$-equivariant  bijection of this set onto the set $\Lsr$ introduced in Example \ref{inforWo} and we will identify the two. So now each $l\in \Lsr$ can  be thought of as  the connected component of $\hat\Sigma_\edge$.  Recall that we there defined a surjection $\wedge^2(\Zds^\Lsr)\to \wedge^2 W_o$ with kernel  $(\sum_{l\in \Lsr} l)\wedge \Zds^\Lsr\cong W_o$ making  the following sequence of $\Zds\Isr$-modules is exact
\begin{equation}\label{wedWo}
	0\to W_o\to \wedge^2 \Zds^{\Lsr}\to  \wedge^2W_o\to 0
\end{equation}
By taking the dual of this exact sequence, we get a sequence that is still exact
\begin{equation}\label{dualwedgwo}
	0\to (\wedge^2W_o)^{\vee}\to (\wedge^2 \Zds^\Lsr)^{\vee}\to W_o^{\vee}\to 0
\end{equation}

\begin{lem}\label{inftL}
	The natural homotopy class of maps $\Sigma_\edge\to G_\edge$ induces an isomorphism on $H^1$  and
	the  map which assigns to the ordered distinct pair $(l, l')$ in $\Lsr$ the $1$-cocycle on $G_\edge$ spanned by
	the vertices defined by $l$ and $l'$ induces an $\Isr$-equivariant isomorphism  $\wedge^2W_0\cong H^1(G_\edge)$. If we recall that $H^1(G_\edge)$ is naturally identified with the vanishing homology of the degeneration $\Sigma$ into $\Sigma_\edge$, then this isomorphism identifies the set $\Delta_\trc$ of vanishing cycles with  the
	set of unordered distinct pairs in $\Lsr$.

	Dually, $L_\edge^{\vee}=H_1(\Sigma_\edge)$ is as a $\Zds\Isr$-module isomorphic to $\wedge^2 W_o^{\vee}$.
\end{lem}
\begin{proof}
	We have already observed that  $\Sigma_\edge \to G_\edge$ induces an isomorphism on $H^1$. Let us now recall
	how $H^1(G)=Z^1(G)/B^1(G)$ is computed for any finite graph $G$: the group of $1$-cocycles $Z^1(G)$ has as its generators the oriented edges (with the understanding that orientation reversal gives the opposite orientation) and the coboundary map assigns to a vertex the sum of the oriented edges which have that vertex as their source. If we do this for $G_\edge$, which is the complete graph on the set $\Lsr$, we see that $C^0(G_\edge)=\Zds^\Lsr=W$, $Z^1(G_\edge)=\wedge^2\Zds^\Lsr=\wedge^2W$ and the coboundary $\delta: C^0(G_\edge)=Z^1(G_\edge)$ is the map which assigns to $l\in \Lsr$ the sum $\sum_{l'\not=l} l'\wedge l=(\sum_{l'} l')\wedge l$. Hence
	$H^1(G_\edge)=\wedge^2 W_o$. The assertion regarding the vanishing cycles is clear.

	The universal coefficient theorem implies that $H_1(G_\edge)$ is naturally identified with the dual of
	$\wedge^2 W_o$. This is easily seen to be $\wedge^2 W_o^{\vee}$.
\end{proof}

From this Lemma the short exact sequence \eqref{0HHH0} becomes the following sequence of $\Zds\Isr$-modules.
\begin{equation}\label{WHW}
	\xymatrix{
	0\ar[r]& \wedge^2 W_o\ar[r]\ar[dr]_-{j_\edge}& H_1(\Sigma)\ar[r]& (\wedge^2 W_o)^\vee\ar[r]& 0\\
	&&Z_1\ar[u]&&
	}
\end{equation}
Note that $\wedge^2 W_o$ has a single generator as a  $\Zds\Isr$-module, for example $\bar{l} \wedge \bar{l'}$ with $l,l'$ distinct. We have an $\Isr$-isomorphism $\iota_\edge:\wedge^2 W_o\to Z_1$  which sends $\bar{l}\wedge \bar{l'}$ to the element in $V_{\edge}$ with the same stabilizer.
Let us apply the left exact functor $\Hom_{\Zds\Isr}(V_o,\cdot)$ to the short exact sequence \eqref{WHW} and combine it with the exact sequence \eqref{seqph1}
\begin{equation}\label{hominf}
	\xymatrix@C=0.5pc@R=2pc{
	0\ar[r]& \Hom_{\Zds\Isr}(V_o,\wedge^2 W_o)\ar[r]\ar[dr]_-{j_{\edge*}}&\Hom_{\Zds\Isr}(V_o,H_1(\Sigma))\ar[r]^-{\phi}&\Hom_{\Zds\Isr}(V_o, \wedge^2 W_o^{\vee})\ar[r]& \Ext_{\Zds\Isr}(V_o,\wedge^2 W_o)\\
	&&\Hom_{\Zds\Isr}(V_o,Z_1)\ar[u]&&
	}
\end{equation}
As in Proposition \ref{propcwseq} the vertical arrow $\Hom_{\Zds\Isr}(V_o,Z_1)\to \Hom_{\Zds\Isr}(V_o,H_1(\Sigma))$ is injective.

\begin{proof}[Proof of Theorem \ref{monoSig}]
	We proved that  $\Hom_{\Zds\Isr}(V_o,H_1(\Sigma))$ is freely generated by $\frac{1}{3}u_\edge$ and $u_\trc$ as $\Zds$-module in Proposition \ref{propcwseq}.  By the above description, their images lie $L_\edge$ resp.\  $L_\trc$. Hence $\rho_\edge$ resp. $\rho_\trc$ leaves $u_\edge$ resp. $u_\trc$ invariant.

	According to the Picard-Lefschetz formula
	\begin{equation*}
		\rho_\trc(\tfrac{1}{3}u_\edge)(\bar{e})-\tfrac{1}{3}u_\edge(\bar{e})=\sum_{\delta\in \Delta_{\trc}/\{\pm 1\}}\tfrac{1}{3}\la [u_{\edge}(\bar{e})],\delta\ra \delta=
		\sum_{x\in \Csr_0(D)/ \{\iota\}}\tfrac{1}{3}\la [u_{\edge}(\bar{e})],\delta_x\ra \delta_x.
	\end{equation*}
	If $x\in \Csr_0(D)$, then by Lemma \ref{speint}, $\frac{1}{3}\la [u_{\edge}(\bar{e})],\delta_x\ra$ equals $1$ when $\delta$ lies in $x\in E$,  $-1$ when $x\in \iota E$ and zero otherwise. So the sum on the right hand side is $\tfrac{1}{3}\sum_{x\in E} \delta_x =u_\trc(\bar{e}_5)$.
	This gives the asserted matrix representation.

	Similarly,  we have
	\begin{equation*}
		\rho_\edge(u_\trc(\bar{e}))-u_\trc(\bar{e})=\sum_{\delta\in \Delta_\edge/\{\pm1\}}\la [u_{\trc}(\bar{e})],\delta\ra \delta=
		\sum_{y\in \Csr_1(D)/ \{\iota\}}\la [u_{\trc}(\bar{e})],\delta_y\ra \delta_y.
	\end{equation*}
	Then Lemma \ref{speint} implies that $\la [u_{\trc}(\bar{e})],\delta_y\ra $ equals $-1$ when $in(y)\in E$,  $1$ when $in(y)\in \iota E$ and zero otherwise. Hence the right hand side is $-u_\edge(\bar{e})=-3. \tfrac{1}{3}u_\edge(\bar{e})$. The matrix representation follows.
\end{proof}

A short exact sequence \eqref{0HHH0} may not split as $\Zds\Isr$-module. However in these two cases we also have the following  interesting result.

\begin{thm}\label{spltSig}
	Let $L$ stand  for $L_\edge$ or $L_\trc$. Then  the natural map $\phi:\Hom_{\Zds\Isr}(V_o,H_1(\Sigma))\to\Hom_{\Zds\Isr}(V_o, L)$ is onto so that  the exact sequences \eqref{hom275} and \eqref{hominf} split to give an exact sequence
	\begin{equation*}
		0\to \Hom_{ \Zds\Isr}(V_o,L)\to \Hom_{ \Zds\Isr}(V_o,H_1(\Sigma,\Zds))\to \Hom_{ \Zds\Isr}(V_o,L^{\vee})\to 0.
	\end{equation*}
\end{thm}
\begin{proof}
	This sequence is indeed exact after tensoring with $\Cds$.
	Since the middle term is the free abelian group generated by $u_\trc$ and $\frac{1}{3}u_\edge$, we can write any $u'\in \Hom_{ \Zds\Isr}(V_o,L^{\vee})$ as a complex multiple ($\lambda\in\Cds$, say) of the image of
	the image of $u_\trc$ in $\Hom_{ \Zds\Isr}(V_o,L_\edge^{\vee})$ (when $L=L_\edge$) resp.\ of \ $\frac{1}{3}u_\edge$ in
	$\Hom_{ \Zds\Isr}(V_o,L_\trc^{\vee})$ (when $L=L_\trc$). The theorem amounts to the assertion that then $\lambda\in\Zds$.
	Lemma \ref{speint} shows that `taking the intersection product with $u_\trc(\bar e)$ resp.\   $\frac{1}{3}u_\edge(\bar e)$' defines a linear form on $L_\edge$ resp.\ $L_\trc$ that  takes the value
	$1$. Hence $\lambda\in\Zds$.
\end{proof}

\subsection{The Local Monodromy Near Irreducible Curve with 6 Nodes}
For the parameter $s=-1$ in $\Bsr$, the associated curve $C_{-1}$ is irreducible with 6 nodes and its normalization $\hat{C}_{-1}$ is irreducible of genus 4. It is the unique smooth non-hyperelliptic $\Isr$-curve of genus 4 which is called the Bring's curve (see Chapter 5 of \cite{Cheltsov2015Cremona}). Its dual intersection graph {$G_s$} has only one vertex and six edges with the vertex marked with 4. In this case the kernel $L$ in the exact sequence  \eqref{kHH} is generated by 6 elements  and is a $\Zds\Isr$-module  a copy of $E$.
Since $\Hom_{\Zds\Isr}(V_o,E)=0$ it follows that $\Hom_{\Zds\Isr}(V_o,\cdot)$ applied  to the exact sequence \eqref{kHH} gives the exact sequence
\begin{equation*}
	0\to \Hom_{\Zds\Isr}(V_o,H_1(C_t))\to \Hom_{\Zds\Isr}(V_o,H_1(C_{-1}))\to\Ext_{\Zds\Isr}^1(V_o, L)
\end{equation*}
Since the map $\Hom_{\Zds\Isr}(V_o,H_1(C_t))\to \Hom_{\Zds\Isr}(V_o,H_1(C_{-1}))$ is invariant under the monodromy action,  we find

\begin{thm}\label{thm:triv}
	The monodromy of the degeneration  at $C_1$ acts trivially on $\Hom_{\Zds\Isr}(V_o,H_1(C_t))$.
\end{thm}

This  completely determines the monodromy representation $\rho_{V_o}$:

\begin{cor}\label{cor:}
	In terms of the  basis $(U_\edge, U_\trc)$ of $\Hom_{\Zds\Isr}(V_o,H_1(C_{-1}))$ the monodromy $\rho_{V_o}$ takes the simple loops around
	$\infty$ and $\frac{27}{5}$  to
	\[
		\rho_\edge=\begin{pmatrix} 1 & -3\\ 0 & 1\end{pmatrix} \text{ resp. } \rho_\trc=\begin{pmatrix} 1  & 0\\ 1& 1\end{pmatrix},
	\]
	is trivial around the puncture $s=-1$ and hence is given around the puncture $s=0$ by
	\[
		\rho_{0}:=(\rho_\trc\rho_\edge)^{-1}=\begin{pmatrix} -2 & 3\\ -1 & 1\end{pmatrix}.
	\]
	In particular, $\rho_{0}$ is of order $3$.
	The image of $\rho_{V_o}$ preserves the sublattice $3U_\edge +U_\trc$, in other words it takes its values in the subgroup $\Gamma_1(3)$ of matrices $(\begin{smallmatrix}  a&b\\ c& d\end{smallmatrix})\in \SL(2, \Zds)$ with $b\in 3\Zds$ and $a\equiv d\equiv 1\pmod{3}$.
\end{cor}

\begin{rmk}\label{rem:repbas}
	if we replace our ordered basis by $(U_\trc, -U_\edge)$, then in the description of  $\Gamma_1(3)$ the condition  $b\in 3\Zds$  becomes $c\in 3\Zds$, which gives the more customary definition $\Gamma_1(3)$.
\end{rmk}

\subsection{The Local Monodromy Near the Triple Conic}
By way of a  check on our computations, we now also give a geometric proof of the fact that the monodromy around $s=0$ is of order three.
Remember that $C_{0}$ is the unstable curve $3K$, where $K$ is a $\Isr$-invariant (smooth) conic. Let $U\subset \Bsr$ be an open  disk centered at $s=0$ of radius $<27/5$.
We proved in \cite{zi2021geometry} that by doing a base change over $U$ of order $3$ (with Galois group $\mu_3$), given by
$\hat t\in \hat U\mapsto t=\hat t^3\in U$,  the  pull back of $\Wsr_U/\hat U$ can be modified over the central fiber $C_0$ only
to make it a smooth family $\hat W_{\hat U}/\hat U$ which still retains the $\mu_3$-action. The central fiber is then a smooth curve
$\hat C_0$ with an action of $\Isr\times \mu_3$ whose $\mu_3$-orbit space gives $K$. This implies that
the monodromy of the original family around $0$ (which is a priori only given as an isotopy class of diffeomorphisms of a nearby smooth fiber) can be represented by the action of a generator $\phi\in\mu_3$ on $\hat C_0$ (which indeed commutes with the $\Isr$-action on $C_0$).

\begin{cor}
	Let $t\in U-\{0\}$. The monodromy automorphism acts on $H_1(\hat C_0; \Cds)$ with eigenvalues of order $3$ only.
	In particular it acts on $\Hom_{\Zds \Isr}(V_o,H_2(\hat C_0))$ with order $3$.
\end{cor}
\begin{proof}
	The subspace of $H_1(\hat C_0; \Cds)$ on which $\phi$ acts as the identity maps isomorphically onto $H_1(C_0; \Cds)=H_1(K; \Cds)=0$
	and hence $\phi$ has eigenvalues of order $3$ only.
\end{proof}

\section{Global Monodromy and the  Period  Map: proof of Theorem \ref{mtm2}}
We first show that $\rho_\edge$ and $\rho_\trc$ generate  all of $\Gamma_1(3)$. This is in fact known, but let us outline a proof nevertheless.
Let us first observe that the natural map $\SL(2,\Zds)\to \SL(2,\Fds_3)$ is onto and that $\Gamma_1(3)$ is the preimage of the subgroup
of $ \SL(2,\Fds_3)$ which fixes the second basis vector of $\Fds_3^2$. Since the  $ \SL(2,\Fds_3)$-orbit of that vector consists of all nonzero elements of $\Fds_3^2$, its index is $8$.

Now recall that $\SL(2,\Zds)$ is generated by
\[
	T:=\begin{pmatrix} 1 & 1\\ 0 & 1\end{pmatrix} \text{  and  } S:=\begin{pmatrix} 0  & 1\\ -1& 0\end{pmatrix},
\]
and that a set of defining relations is $S^4=(ST)^6=1$ and $S^2T=TS^2$. The element $S^2$ represents $-1$ and generates the kernel of the projection
$\SL(2,\Zds)/\{\pm 1\}\to \PSL(2,\Zds)$. So this gives also a presentation of $\PSL(2,\Zds)$: if $\overline S$ resp.\ $\overline T$ denotes their images in $\PSL(2,\Zds)$,
then a set of defining relations for $\PSL(2,\Zds)$ becomes $\overline S^2=(\overline S\, \overline T)^3=1$.

Note that $\rho_\edge=T^{-3}$ and $\rho_\trc=ST^{-1}S^{-1}$  and so our  monodromy group
$\Gamma$ is the subgroup of $\SL(2,\Zds)$ generated by $T^3$ (which defines defines a cusp of width $3$)  and $T':=STS^{-1}$ (which defines a cusp of width $1$).

\begin{prop}\label{prop:}
	The  index of $\Gamma$ in $\SL(2,\Zds)$ is $8$ so that $\Gamma=\Gamma_1(3)$. In fact the left cosets of $\Gamma$ in $\SL(2,\Zds)$ are represented by
	by $\{1,T,T^{-1},S\}\cup \{1, T, T^{-1}, S\}S^2$.
\end{prop}
\begin{proof}
	Since $S^2=-1$ is central, it suffices to show
	that the corresponding statement holds for the subgroup $\overline\Gamma$ in $\PSL(2, \Zds)$  generated by $\overline T$ and
	$\overline T'= \overline S\,\overline T\, \overline S^{-1}$: we then must show that $\{1,\overline T, \overline T^{-1}, \overline S\}$ represent its left cosets. In other words, we must show that this set is invariant under left multiplication by $\overline T$ or $\overline S$. This follows from the identities
	\begin{itemize}
		\item[] $\overline T\, \overline S=\overline S\, \overline T'\in \overline S\, \overline\Gamma$,
		\item[] $\overline S\, \overline T=(\overline S\, \overline T)^{-2}=\overline T^{-1}\overline S^{-1}\overline T^{-1}\overline S^{-1}=\overline T^{-1}\overline S \overline T^{-1}\overline S=\overline T^{-1}\overline T'^{-1}\in \overline T^{-1}\overline\Gamma$, and
		\item [] $\overline S\, \overline T^{-1}=(\overline T\, \overline S)^{-1}=\overline T\, \overline S\,\overline T\, \overline S=\overline T\,\overline T'\in \overline T\,\overline\Gamma$.\qedhere
	\end{itemize}
\end{proof}
\begin{figure*}
	\centering
	{\includegraphics[width=1.0\textwidth]{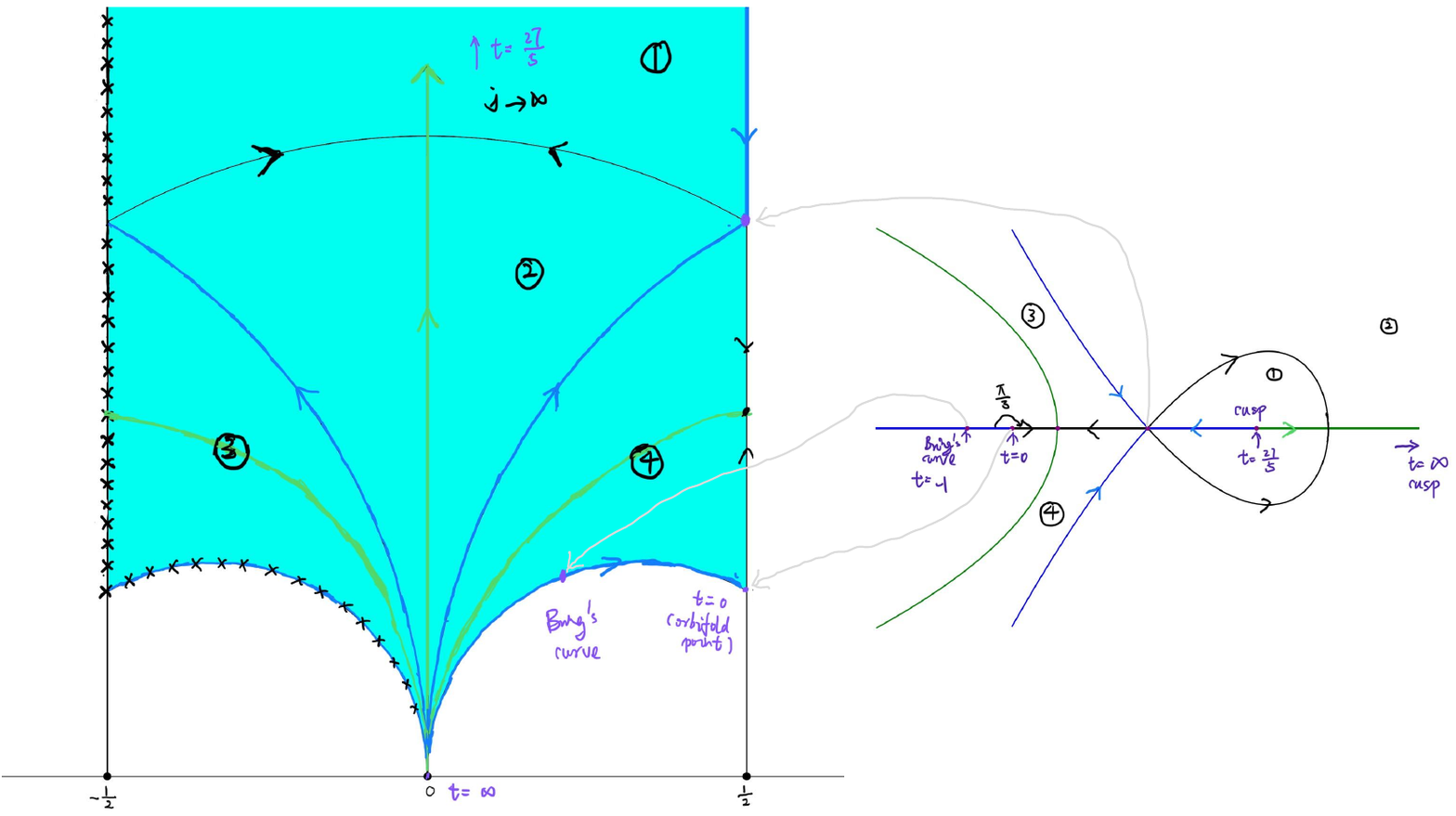}}
	\caption{\small{Left:\ Fundamental Domain of $\Gamma_1(3)$;\ Right:\ Picture of $\Bsr$.
			In the quotient the boundary is identified by reflection with respect to the imaginary axis.
			The cusp at $\infty$ has width $1$ and represents the irreducible curve with ten nodes, the cusp at $0$
			has width $3$ and  represents the sum of six lines and  $\frac{3+\sqrt{3}i}{6}$ represents the triple conic.}}
	\label{fuddGamma}
\end{figure*}

For what follows we need  a few facts regarding $\Gamma_1(3)$. Let us, as is customary, write $X_1(3)$ for  $\Gamma_1(3)\backslash  \Hsr$. This is a modular curve $X_1(3)$ which classifies  elliptic curves endowed with a point of order $3$. Then $X_1(3)$ is completed to a smooth projective curve $Y_1(3)$ by filling in a finite number cusps. These are by definition the orbits of $\Gamma_1(3)$ in $\Pds^1(\Qds)$
and in this case they consist of two elements and make $Y_1(3)$ isomorphic to $\Pds^1$ (one with cusp width $3$ and the other with cusp width $1$). The action of $\Gamma_1(3)$ on $\Hsr$ is free except for one orbit $\Osr$: that orbit consists of the points having a stabilizer of order $3$. This accounts for an orbifold point $o$ of $X_1(3)$ of order $3$.
\\

Let $\widetilde\Bsr^\circ\to \Bsr^\circ$ be the  $\Gamma$-cover. So given $t\in\Bsr^\circ$, then a point $\tilde t$ of $\widetilde\Bsr^\circ$ over $t$ can be represented by a path $\beta:[0,1]\to \Bsr^\circ$ with $\beta(0)$ on the real interval $(\frac{27}{5}, \infty)$ and $\beta(1)=t$. Such a path can be used to transport the basis $(U_\edge, U_\trc)$ of $\Hom_{\Zds\Isr}(V_o,H_1(\Sigma))$ to  a basis of  $\Hom_{\Zds\Isr}(V_o,H_1(C_t))$. It follows from the definitions that another such path
defines the same basis of $\Hom_{\Zds\Isr}(V_o,H^1(C_t))$ if and only if it defines the same point $\tilde t$ of $\widetilde\Bsr^\circ$ over $t$.
We therefore denote that basis  $(U_\edge(\tilde t), U_\trc(\tilde t))$.

The discussion in Subsection \ref{subsect:isotypical} shows  (in a much more general setting) that the symplectic form on
$H_1(C_t; \Qds)$  and
the inner product on $V_\Qds$ determine a symplectic form on the isogeny space $\Hom_{\Zds\Isr}(V_\Qds, H_1(C_t; \Qds))$.
Let us endow $V_0$ with the trivial Hodge structure of bidegree $(0,0)$. This is polarized by the $\Qds$-valued inner product that we described in Example \ref{inforVo}. Then the Hodge structure on $H^1(C_t)$ polarized by the intersection form then determines a Hodge structure on  $\Hom_{\Zds\Isr}(V_o,H^1(C_t))$ with as only nonzero Hodge numbers $h^{1,0}=h^{0,1}=1$  and polarized by the standard symplectic form. Its basis $(U_\edge(\tilde t), U_\trc(\tilde t))$  defines a point in the upper half plane $\Hsr$ as follows: the complex vector space $\Hom_{\Cds\Isr}(V,H^{1,0}(C)))= \Hom_{\Cds\Isr}(V,H^0(C, \Omega_C)))$ is of dimension one and if $\omega$ is a generator, then
this point is given by $\omega (U_\trc(\tilde t))/\omega(U_\edge(\tilde t))\in \Hsr$. We thus obtain a $\Gamma$-equivariant holomorphic map
\[
	\widetilde \Psr^\circ_V: \widetilde\Bsr^\circ\to \Hsr.
\]

Let $\Bsr^+$ be obtained from $\Bsr$ by filling in the point $0$ that represents the triple conic and $-1$ that represents the irreducible curve with 6 nodes, so that $\Bsr^+=\Pds^1-\{\frac{27}{5}, \infty\}$. Since the monodromy of the Winger pencil around $t=0$ has order $3$, the $\Gamma$-covering $\widetilde\Bsr^\circ\to \Bsr^\circ$ extends
to a $\Gamma$-covering $\widetilde\Bsr^+\to \Bsr^+$ that has ramification index $3$ over $0$. Then the map  $\widetilde \Psr^\circ_V$
extends to a $\Gamma$-equivariant map
\[
	\widetilde \Psr^+_V: \widetilde\Bsr^+\to \Hsr,
\]
which therefore induces  a holomorphic map
\[
	\Psr^+_V: \Bsr^+\to X_1(3).
\]
As is well-known, the fact that the period map has regular singularities implies that this map  lives in the algebraic category. Hence it extends (uniquely) to a  morphism between the completions of source and target:
\[
	\Psr_V: \Bsr\to Y_1(3).
\]

The main result of this section is the following, which can be understood as a strong Torelli theorem for our family. It is a reformulation
of Theorem \ref{mtm2}.

\begin{thm}\label{thm:main}
	The map $\Psr_V$ is an isomorphism which takes the point $t=0$  represented by the triple conic to the unique orbifold point of $X_1(3)$ of degree $3$, whereas the points $t=\infty$ and $t=\frac{27}{5}$ are mapped to the cusps of width $3$ and $1$ respectively.
\end{thm}
\begin{proof}
	The map $\Psr_V: \Bsr\to Y_1(3)$ is a nonconstant map between connected complete complex curves and hence it must be surjective and of finite degree. We must show that the degree (that we shall denote by $d$) is $1$.
	The $\Psr_V$-preimages of the two  cusps of $Y_1(3)$ must be $\{\frac{27}{5},\infty\}$ so that we will have a total ramification over each cusp. Both $\Bsr$ and $Y_1(3)$ are copies of $\Pds^1$ and so have euler characteristic $2$. The Riemann-Hurwitz formula then implies that there cannot be any further ramification, for these two ramification points already bring down the euler characteristic of the total space to
	the desired number: $2d-(d-1)-(d-1)=2$. This implies that $ \Psr^+_V$ is a local isomorphism.

	We claim that $\Psr_V^{-1}(o)=0$.  This will imply the theorem, for it  then  follows that $d=1$.

	It is clear that $\Psr_V(0)=o$.  On the other hand, the composite map $\widetilde\Bsr^\circ\xrightarrow{ \widetilde \Psr^+_V}\Hsr \to X_1(3)$ is equal to the composite map $\widetilde\Bsr^\circ\to \Bsr^\circ \to X_1(3)$
	and since both $\widetilde\Bsr^\circ\to \Bsr$ and $\Bsr^\circ \to X_1(3)$ are local isomorphisms, so is $\widetilde\Bsr^\circ\to \Hsr \to X_1(3)$. As $\Hsr \to X_1(3)$  ramifies over $o$, it follows  that this composite must land in
	$X_1(3)-\{ o\}$. Hence $\Psr^+_V(\Bsr^\circ)\subset X_1(3)-\{ o\}$.
\end{proof}

\begin{rmk}\label{rem:Eo}
	The preceding theorem does not follow from the usual Torelli theorem since it only concerns the $V_o$-part of the cohomology.
	There is a similar period map for the $E_o$-part of the monodromy. The corresponding monodromy group will land in an arithmetic subgroup of a group isomorphic to $\SL(2,\Qds(\sqrt{5}))$ and the period map will land in a Hilbert modular surface associated to this group. If we combine this with the  above theorem, we then find a map from $X_1(3)$ to this Hilbert modular surface. Note that the point representing Bring's curve is mapped to a cusp point of the Hilbert modular surface and so this cannot be a morphism of  Shimura varieties.
\end{rmk}

\bibliographystyle{spmpsci}
\bibliography{./Reference}

\end{document}